\documentclass[a4paper, 12pt, oneside]{amsart}

\usepackage[T1]{fontenc}
\usepackage[utf8]{inputenc}
\usepackage{mathptmx}       
\usepackage{helvet}         
\usepackage{courier}        
\usepackage{type1cm}        
\usepackage{graphicx}        

\usepackage{amssymb, amsmath}
\usepackage{enumitem}
\usepackage[DIV=10]{typearea}
\usepackage{tikz}

\newcommand{\N}{\mathbb{N}}

\newcommand{\G}{\mathcal{G}}

\newcommand{\reg}{{\operatorname{reg}}}
\newcommand{\sing}{{\operatorname{sing}}}
\newcommand{\lp}{{\operatorname{lp}}}

\theoremstyle{theorem}
\newtheorem{theorem}{Theorem}[section]
\newtheorem{proposition}[theorem]{Proposition}
\newtheorem{corollary}[theorem]{Corollary}
\newtheorem{lemma}[theorem]{Lemma}

\theoremstyle{definition}

\newtheorem{remark}[theorem]{Remark}
\newtheorem{example}[theorem]{Example}

\title{Orbit equivalence of graphs and isomorphism of graph groupoids}

\author{Toke Meier Carlsen}
\address{Department of Science and Technology\\University of the Faroe Islands\\
N\'oat\'un 3\\ FO-100 T\'orshavn\\ Faroe Islands}
\email{toke.carlsen@gmail.com}

\author{Marius Lie Winger}
\address{Department of Teacher Education 1-7 and BA in Archive and Collection Management\\Norwegian University of Science and Technology\\Faculty of Teacher and Interpreter Education\\Postboks 8900\\NO-7491 Trondheim\\Norway}
\email{marius.l.winger@ntnu.no}

\date{\today}

\begin{document}

\maketitle

\begin{abstract}
	We show that the groupoids of two directed graphs are isomorphic if and only if the two graphs are orbit equivalent by an orbit equivalence that preserves isolated eventually periodic points. We also give a complete description of the (topological) isolated points of the boundary path space of a graph. As a result, we are able to show that the groupoids of two directed graphs with finitely many vertices and no sinks are isomorphic if and only if the two graphs are orbit equivalent, and that the groupoids of the stabilisations of two such graphs are isomorphic if and only if the stabilisations of the graphs are orbit equivalent.
\end{abstract}

\section{Introduction}

From a directed graph $E$, one can construct an ample groupoid $\G_E$ such that the $C^*$-algebra $C^*(\G_E)$ of $\G_E$ is isomorphic to the graph $C^*$-algebra $C^*(E)$, and, for any commutative ring $R$ with a unit, the Steinberg algebra $A_R(\G_E)$ is isomorphic to the Leavitt path algebra $L_R(E)$. This has recently been used in \cite{ABHS}, \cite{BCH}, \cite{BCW}, \cite{C}, \cite{CR}, and \cite{CRS} to explore the relationship between isomorphisms of graph $C^*$-algebras, isomorphisms of Leavitt path algebras, and isomorphisms of graph groupoids. It is therefore useful to be able to characterise when the groupoids of two directed graphs are isomorphic. 

In \cite{BCW}, the notion of \emph{orbit equivalence} of directed graphs was introduced as an analogue of continuous orbit equivalence of shifts of finite type \cite{Mat}, and it was shown that if the groupoids $\G_E$ and $\G_F$ of two directed graphs $E$ and $F$ are isomorphic, then the graphs $E$ and $F$ are orbit equivalent. It was also shown that if every cycle in $E$ and $F$ has an exit, then the converse holds, and that the converse does not hold in general. 

In \cite{AER}, the relationship between orbit equivalence of graphs, isomorphism of graph $C^*$-algebras, and isomorphism of graph groupoids was further explored, and by using the connection with graph $C^*$-algebras, it was shown that if $E$ and $F$ are two directed graphs, then the graph groupoids $\G_E$ and $\G_F$ are isomorphic if and only if $E$ and $F$ are orbit equivalent by an equivalence that maps isolated eventually periodic points to isolated eventually periodic points. 

In this paper, we prove the above mentioned result from \cite{AER} directly without using $C^*$-algebras. We also give a complete description of the topological isolated points of the boundary path space of a graph. As a result, we are able to show that if $E$ and $F$ are directed graphs with finitely many vertices and no sinks, then the groupoids of $E$ and $F$ are isomorphic if and only if $E$ and $F$ are orbit equivalent, and the groupoids of the stabilisations of $E$ and $F$ are isomorphic if and only if the stabilisations of $E$ and $F$ are orbit equivalent.

The rest of the paper is organised in the following way. In Section 2 we recall the definitions of a directed graph, the boundary path space of a directed graph, orbit equivalence of directed graphs, and the groupoid of a directed graph; and introduce notation. In Section 3 we characterise all (topological) isolated points in the boundary path space of a directed graph (Proposition~\ref{pro:a}) and illustrate the different types of isolated points in Example \ref{exam}. Finally, in Section 4 we state and prove our main result (Theorem~\ref{thm:b}) from which it follows that the groupoids of two directed graphs are isomorphic if and only if the two graphs are orbit equivalent by an orbit equivalence that preserves isolated eventually periodic points, and Corollary~\ref{cor:1} that says that the groupoids of two directed graphs with finitely many vertices and no sinks are isomorphic if and only if the two graphs are orbit equivalent, and that the groupoids of the stabilisations of two such graphs are isomorphic if and only if the stabilisations of the graphs are orbit equivalent.

\section{Preliminaries}

For the benefit of the reader we recall in this section the definitions of a directed graph, the boundary path space of a directed graph, orbit equivalence of directed graphs, and the groupoid of a directed graph; and we introduce notation. Most of this is standard and can be found in many other papers, for example \cite{AER}, \cite{BCW}, and \cite{CR}.

We let $\N$ denote the set of nonnegative integers (so $0\in\N$).

\subsection{Directed graphs}
A \emph{directed graph $E$} is a tuple $(E^0, E^1, r, s)$ where $E^0$ and $E^1$ are countable sets and $r$ and $s$ are functions from $E^1$ to $E^0$. Elements in $E^0$ are called \emph{vertices} and elements in $E^1$ are called \emph{edges}. The functions $r$ and $s$ are called the \emph{range function} and the \emph{source function} respectively, and for an edge $e\in E^1$ the vertices $r(e)$ and $s(e)$ are called the range and the source of $e$ respectively.

By forming a sequence of edges $\mu_1, \mu_2, \ldots, \mu_n$ in $E$ such that $r(\mu_i) = s(\mu_{i+1})$ for $1 \leq i \leq n-1$ we get a \emph{path} $\mu = \mu_1\mu_2\cdots\mu_n$. We denote by $|\mu|$ the length of the path $\mu$, and form the set $E^n$ consisting of all paths of length $n$. This notation is consistent with $E^0$ by considering vertices as paths of zero length. We can naturally extend the source and range maps to paths by setting $r(\mu) := r(\mu_n)$ and $s(\mu) := s(\mu_1)$. For any $v \in E^0$, we set $s(v) := r(v) := v$. If $v \in E^0$ and $n \in \N$, then $v E^n := \{ x \in E^n : s(x) = v \}$ and $E^n v := \{ x \in E^n : r(x) = v \}$. Further we let $E^*$ be the collection of all paths of finite length in $E$, i.e. $E^* = \bigcup_{n \in \N} E^n$. 

A vertex $v\in E^0$ is a \emph{sink} if $vE^1=\emptyset$, and an \emph{infinite emitter} if $vE^1$ is infinite. We define the \emph{singular vertices} of $E$ as $E_\sing^0 := \{v\in E^0:v\text{ is a sink or an infinite emitter}\}$ and the \emph{regular vertices} of $E$ as $E_\reg^0 := E^0\setminus E_\sing^0$.

A \emph{cycle} (sometimes called a \emph{loop} in the literature) in $E$ is a path $\gamma\in E^*$ such that $|\gamma|\ge 1$ and $s(\gamma)=r(\gamma)$. An edge $e\in E^1$ is an \emph{exit} to the loop $\gamma=\gamma_1\dots\gamma_{|\gamma|}$ if there exists $i$ such that $s(e)=s(\gamma_i)$ and $e\ne\gamma_i$. The graph $E$ is said to satisfy \emph{condition (L)} if every loop in $E$ has an exit. 

\subsection{The boundary path space}
An \emph{infinite path} is an infinite sequence $x = (x_i)_{i \in \N}$ where $x_i \in E^1$ and $r(x_i) = s(x_{i+1})$ for all $i \in \N$. We set $s(x) := s(x_0)$ and denote by $E^\infty$ the collection of all infinite paths in $E$.

The \emph{boundary path space $\partial E$} of $E$ is defined as 
$$\partial E := E^\infty \cup \{ \mu \in E^* : r(\mu) \in E_\sing^0 \}.$$
If $\mu=\mu_1\mu_2\cdots\mu_m\in E^*$, $x=x_1x_2\cdots\in E^\infty$ and $r(\mu)=s(x)$, then we let $\mu x$ denote the infinite path  $\mu_1\mu_2\cdots\mu_m x_1x_2\cdots \in E^\infty$ (if $\mu\in E^0$, then $\mu x=x$). 

For $\mu\in E^*$, the \emph{cylinder set} of $\mu$ is the set
\begin{equation*}
	Z(\mu):=\{\mu x\in\partial E:x\in r(\mu)\partial E\},
\end{equation*} 
where $r(\mu)\partial E:=\{x \in \partial E : r(\mu)=s(x)\}$. Given $\mu\in E^*$ and a finite subset $F\subseteq r(\mu)E^1$ we define
\begin{equation*}
	Z(\mu\setminus F):=Z(\mu)\setminus\left(\bigcup_{e\in F}Z(\mu e)\right).
\end{equation*}

The boundary path space $\partial E$ is a locally compact Hausdorff space with the topology given by the basis $\{Z(\mu\setminus F): \mu\in E^*,\ F\text{ is a finite subset of }r(\mu)E^1\}$, and each such $Z(\mu\setminus F)$ is compact and open (see \cite[Theorem 2.1 and Theorem 2.2]{Web}).

\subsection{Orbit equivalence}
For $n\in\N$, let $\partial E^{\ge n}:=\{x\in\partial E: |x|\ge n\}$. Then $\partial E^{\ge n}= \cup_{\mu \in E^n} Z(\mu)$ is an open subset of $\partial E$. We define the \emph{shift map} on $E$ to be the map $\sigma_E:\partial E^{\ge 1}\to\partial E$ given by $\sigma_E(x_1x_2x_3\cdots)=x_2x_3\cdots$ for $x_1x_2x_3\cdots\in\partial E^{\ge 2}$ and $\sigma_E(e)=r(e)$ for $e\in\partial E\cap E^1$. For $n\ge 1$, we let $\sigma_E^n$ be the $n$-fold composition of $\sigma_E$ with itself. We let $\sigma_E^0$ denote the identity map on $\partial E$. Then $\sigma_E^n$ is a local homeomorphism for all $n\in\N$. When we write $\sigma_E^n(x)$, we implicitly assume that $x\in\partial E^{\ge n}$.

An infinite path $x\in E^\infty\subseteq\partial E$ is said to be \emph{eventually periodic} if there are $n,p\in\N$ with $p>0$ such that $\sigma_E^{n+p}(x)=\sigma_E^n(x)$. For an eventually periodic infinite path $x$, denote by $\lp(x)$ the \emph{least period of $x$}, i.e. 
\begin{equation*}
\lp(x) := \min\{ p > 0 : \text{ there exists } m,n \in \mathbb{N} \text{ such that } p = m - n \text{ and } \sigma_E^{n}(x) = \sigma_E^m(x) \}.
\end{equation*}

Notice that $x\in E^\infty\subseteq\partial E$ is eventually periodic if and only if there are a finite path $\mu\in E^*$ and a cycle $\gamma\in E^*$ such that $x=\mu\gamma\gamma\gamma\dots$.   

Let $E$ and $F$ be directed graphs. A homeomorphism $h:\partial E\to\partial F$ is called an \emph{orbit equivalence} if there are continuous functions $k,l:\partial E^{\ge 1}\to\N$ and $k',l':\partial F^{\ge 1}\to\N$ such that
\begin{equation}\label{eq:1}
	\sigma_F^{k(x)}(h(\sigma_E(x)))=\sigma_F^{l(x)}(h(x))\text{ and }\sigma_E^{k'(y)}(h^{-1}(\sigma_F(y)))=\sigma_E^{l'(y)}(h^{-1}(y))
\end{equation}
for all $x\in\partial E^{\geq 1}$ and all $y\in\partial F^{\geq 1}$. The directed graphs $E$ and $F$ are said to be \emph{orbit equivalent} if there is an orbit equivalence $h:\partial E\to\partial F$. 

\subsection{The graph groupoid}
The \emph{graph groupoid} of a directed graph $E$ is the ample Hausdorff groupoid 
\begin{equation*}
	\G_E=\{(x,m-n,y): x,y\in\partial E,\ m,n\in\N,\text{ and } \sigma_E^m(x)=\sigma_E^n(y)\},
\end{equation*}
with product $(x,k,y)(w,l,z):=(x,k+l,z)$ if $y=w$ and undefined otherwise, and inverse given by $(x,k,y)^{-1}:=(y,-k,x)$. The topology of $\G_E$ is generated by subsets of the form $Z(U,m,n,V):=\{(x,k,y)\in\G_E:x\in U,\ k=m-n,\ y\in V,\ \sigma_E^m(x)=\sigma_E^n(y)\}$ where $m,n\in\N$, $U$ is an open subset of $\partial E^{\ge m}$ such that the restriction of $\sigma_E^m$ to $U$ is injective, and $V$ is an open subset of $\partial E^{\ge n}$ such that the restriction of $\sigma_E^n$ to $V$ is injective, and $\sigma_E^m(U)=\sigma_E^n(V)$. 

The map $x\mapsto (x,0,x)$ is a homeomorphism from $\partial E$ to the unit space $\G_E^0$ of $\G_E$. If we identify $\G_E^0$ and $\partial E$ by this homeomorphism, then the source and range maps $s,r:\G_E\to\partial E$ of $\G_E$ are given by $s((x,k,y))=y$ and $r((x,k,y))=x$. We denote by $c_E$ the continuous cocycle from $\G_E$ to $\mathbb{Z}$ given by $c_E((x,k,y))=k$.

A \emph{bisection} of a groupoid $\G$ is a subset $A\subseteq\G$ such that the restrictions of the source and range maps to $A$ are injective.  As shown in \cite{Renault}, each bisection $A$ of an étale groupoid $\G$ defines a homeomorphism $\alpha_A:s(A)\to r(A)$ by $\alpha_A(s(\eta))=r(\eta)$ for $\eta\in A$. Renault defines the \emph{pseudogroup} of an étale groupoid $\G$ to be the inverse semigroup 
$$\{\alpha_A:A\text{ is an open bisection of }\G\}.$$ 
As in \cite{AER} and \cite{BCW}, we write $\mathcal{P}_E$ for the pseudogroup of $\G_E$. For $\alpha\in\mathcal{P}_E$, we write $s(\alpha)$ for the domain of $\alpha$, and $r(\alpha)$ for the range of $\alpha$.

All isomorphisms between groupoids considered in this paper are, in addition to preserving
the groupoid structure, homeomorphisms.

\section{Isolated points of $\partial E$}

In this section we characterise all (topological) isolated points in the boundary path space of a directed graph.  

Let $E$ be a directed graph. A finite word $\mu\in E^*$ belongs to the boundary path space $\partial E$ if and only if $r(\mu)$ is a sink or an infinite emitter, and $\mu$ is an isolated point of $\partial E$ (i.e., $\{\mu\}$ is open in $\partial E$) if and only if $r(\mu)$ is a sink. An infinite path $x\in E^\infty\subseteq \partial E$ is an isolated point of $\partial E$ if and only if the set $\{n\in\N: |r(x_n)E^1|\ge 2\}$ is finite. We have in particular that an eventually periodic point $x=\mu\gamma\gamma\gamma\dots$ is isolated in $\partial E$ if and only if the cycle $\gamma$ does not have an exit. Notice that $E$ satisfies condition (L) if and only if that are no isolated eventually periodic points in $\partial E$.

We say that $x\in E^\infty$ is a \emph{wandering point} if the set $\{n\in\N: s(x_n)=v\}$ is finite for all $v\in E^0$.

We then have the following classification of the isolated points of the boundary path space $\partial E$.

\begin{proposition}\label{pro:a}
	Let $E$ be a directed graph. If $x$ is an isolated point of $\partial E$, then it is either finite (in which case $r(x)$ is a sink), eventually periodic or a wandering point.
\end{proposition}

\begin{proof}
	Suppose that $x$ is an isolated point of $\partial E$, and that it is not finite or a wandering point. Then there is a $v\in E^0$ such that the set $\{n\in\N: s(x_n)=v\}$ is infinite. Choose $n_1<n_2$ such that $s(x_{n_1})=s(x_{n_2})=v$, and set $\mu:=x_0x_1\dots x_{n_1-1}$ and $\gamma:=x_{n_1}\dots x_{n_1+1}\dots x_{n_2-1}$. Then $\mu,\gamma\in E^*$, $r(\mu)=s(\gamma)$, and $\gamma$ is a cycle in $E$. Since $x$ is an isolated point of $\partial E$ and the set $\{n\in\N: s(x_n)=v\}$ is infinite, it follows that $x=\mu\gamma\gamma\gamma\dots$, so $x$ is eventually periodic.
\end{proof}

Let $E$ and $F$ be directed graphs. Since an orbit equivalence $h:\partial E\to\partial F$ is a homeomorphism, it must necessarily map the set of isolated points of $\partial E$ onto the set of isolated points of $\partial F$. The following example shows that the three classes of isolated points mentioned in Proposition \ref{pro:a} can be mixed by an orbit equivalence. It also illustrates that the groupoids of directed graphs that are orbit equivalent, are not necessarily isomorphic (cf.~\cite[Example 5.2]{BCW}).

\begin{example} \label{exam}
	Consider the 3 directed graphs
	\begin{equation*}
		\begin{tikzpicture}
		    \def\vertex(#1) at (#2,#3){
		        \node[inner sep=0pt, circle, fill=black] (#1) at (#2,#3)
		        [draw] {.}; 
		 }
		    \vertex(11) at (0,0)
    
		    \vertex(21) at (0.8,0)
    
		    \vertex(31) at (1.6,0)
    
		    \vertex(41) at (2.4,0)
    
		    \vertex(51) at (4.8,0)
    
		    \vertex(61) at (5.6,0)
    
		    \vertex(71) at (6.4,0)
    
		    \vertex(81) at (7.2,0)
				
		    \vertex(91) at (8.6,0)
    
		    \vertex(101) at (9.4,0)
    
		    \vertex(111) at (10.2,0)
    
		    \vertex(121) at (11,0);
    
		\node at (-1.0,0) {$E$};
		\node at (3.8,0) {$F$};
		\node at (8.2,0) {$G$};

		\node at (4.4,0) {$\dots$};
		\node at (-0.4,0) {$\dots$};
		\node at (11.5,0) {$\dots$};

		\draw[style=semithick, -latex] (11.east)--(21.west);
		\draw[style=semithick, -latex] (21.east)--(31.west);
		\draw[style=semithick, -latex] (31.east)--(41.west);

		\draw[style=semithick, -latex] (51.east)--(61.west);
		\draw[style=semithick, -latex] (61.east)--(71.west);
		\draw[style=semithick, -latex] (71.east)--(81.west);
		
		\draw[style=semithick, -latex] (41.north east)
	        .. controls +(0,0.25) and +(0,0.25) ..
	        +(.5,0)
	        .. controls +(0,-0.25) and +(0,-0.25) ..
	        (41.south east);
		     
 		\draw[style=semithick, -latex] (91.east)--(101.west);
 		\draw[style=semithick, -latex] (101.east)--(111.west);
 		\draw[style=semithick, -latex] (111.east)--(121.west);   
		\end{tikzpicture}
	\end{equation*}
	The boundary path space of each of these 3 directed graphs is homeomorphic to $\N$ equipped with the discrete topology (there is for each vertex exactly one element of the boundary path space that begins at that vertex), and the 3 graphs are orbit equivalent. All the points in $\partial E$ are eventually periodic, all the points in $\partial F$ are finite, and all the points in $\partial G$ are wandering points.
	
	The groupoids $\G_F$ and $\G_G$ are both isomorphic to the discrete groupoid $\N\times\N$ with product defined by $(m_1,m_2)(n_1,n_2):=(m_1,n_2)$ if $m_2=n_1$ and undefined otherwise, and inverse given by $(n_1,n_2)^{-1}:=(n_2,n_1)$. The groupoid $\G_E$ is not isomorphic to $\G_F$ and $\G_G$ because $\{\eta\in\G_E:s(\eta)=r(\eta)=x\}$ is infinite for any $x\in\partial E$, and $\{\eta\in\G_F:s(\eta)=r(\eta)=x\}=\{(x,0,x)\}$ for any $x\in\partial F$.
\end{example}

\section{Orbit equivalence and isomorphism of groupoids}

Let $E$ and $F$ be directed graphs. We say that a homeomorphism $h:\partial E\to\partial F$ \emph{preserves isolated eventually periodic points} if $h$ maps the set of isolated eventually periodic points in $\partial E$ onto the set of isolated eventually periodic points in $\partial F$. In this section we state and prove our main result (Theorem~\ref{thm:b}) from which it follows that the groupoids of two directed graphs are isomorphic if and only if the two graphs are orbit equivalent by an orbit equivalence that preserves isolated eventually periodic points. 

\begin{remark}
	Since a directed graph $E$ satisfies condition (L) if and only if $\partial E$ contains no isolated eventually periodic points, it follows that if two directed graphs $E$ and $F$ both satisfy condition (L), then any homeomorphism $h:\partial E\to\partial F$ preserves isolated eventually periodic points.
\end{remark}

Let $E$ and $F$ be directed graphs and $\phi:\G_E\to\G_F$ be an isomorphism. Then $\phi$ maps $\G_E^0$ onto $\G_F^0$. It follows that there is a unique homeomorphism $h:\partial E\to\partial F$ such that $\phi((x,0,x))=(h(x),0,h(x))$ for all $x\in\partial E$. We call $h$ the \emph{homeomorphism induced by $\phi$}. 

We now present our main result which is a strengthening of \cite[Theorem 3.10]{AER}.

\begin{theorem}\label{thm:b}
	Let $E$ and $F$ be directed graphs. If $\phi:\G_E\to\G_F$ is an isomorphism, then the homeomorphism induced by $\phi$ is an orbit equivalence that preserves isolated eventually periodic points. Conversely, if $h:\partial E\to\partial F$ is an orbit equivalence that preserves isolated eventually periodic points, then there is an isomorphism $\phi:\G_E\to\G_F$ that induces $h$. Thus, $\G_E$ and $\G_F$ are isomorphic if and only if there is an orbit equivalence $h:\partial E\to\partial F$ that preserves isolated eventually periodic points.
\end{theorem}

\begin{remark}
	By combining Theorem~\ref{thm:b} with \cite[Proposition 3.4 and Theorem 5.1]{BCW} (cf. \cite[Proposition 3.8]{AER}) and \cite[Theorem 4.2]{CRS}, we recover \cite[Proposition 3.9, Proposition 4.2, Theorem 4.3, Theorem 5.3, and Corollary 5.4]{AER}.
\end{remark}

For the proof of Theorem~\ref{thm:b}, we need the following lemma which is an adaptation of \cite[Lemma 4.3(2)]{CEOR}.

\begin{lemma}\label{lem:1}
	Suppose $E$ is a directed graph, $\alpha\in\mathcal{P}_E$, and $x\in s(\alpha)$. Then there is an $n\in\mathbb{Z}$ that has the property that there is an open neighbourhood $U$ of $x$ and $k,l\in\mathbb{N}$ such that $U\subseteq s(\alpha)$, $n=l-k$, and $\sigma_E^k(\alpha(x'))=\sigma_E^l(x')$ for all $x'\in U$. If $x$ is not an isolated eventually periodic point, then this $n$ is unique.
\end{lemma}

\begin{proof}
	The existence of an $n$ with the desired properties follows from \cite[Proposition 3.3]{BCW}. To prove uniqueness of $n$ if $x$ is not an isolated periodic point, suppose that $n_1,n_2\in\mathbb{Z}$, $n_1\ne n_2$, $U_1$ and $U_2$ are open neighbourhoods of $x$ such that $U_1,U_2\subseteq s(\alpha)$, and that $k_1,k_2,l_1,l_2\in\mathbb{N}$ such that for $i=1,2$ we have $n_i=l_i-k_i$ and $\sigma_E^{k_i}(\alpha(x'))=\sigma_E^{l_i}(x')$ for all $x'\in U_i$. Let $k_{\max}:=\max\{k_1,k_2\}$, $l'_i:=l_i-k_i+k_{\max}$ for $i=1,2$, and $m:=\min\{l'_1,l'_2\}$ and $p:=\max\{l'_1,l'_2\}-m$. Then $l'_1=n_1+k_{\max}\ne n_2+k_{\max}=l'_2$, so $p>0$. Since
	\begin{equation*}
		\sigma_E^{l'_1}(x')=\sigma_E^{k_{\max}}(\alpha(x'))=\sigma_E^{l'_2}(x')
	\end{equation*}
	for all $x'\in U_1\cap U_2$, it follows that every element of $\sigma_E^m(U_1\cap U_2)$ is periodic with period $p$. In particular, $x$ is eventually periodic. Choose $r>m+p$ such that $Z(x_{[0,r]})\subseteq U_1\cap U_2$. Suppose $x'\in Z(x_{[0,r)})$. Then $x_{[0,m+p)}=x'_{[0,m+p)}$ and both $\sigma_E^m(x)$ and $\sigma_E^m(x')$ are periodic with period $p$. It follows that $x'=x$, and thus that $x$ is an isolated eventually periodic point.
\end{proof}

\begin{proof}[Proof of Theorem~\ref{thm:b}]
	The proof uses ideas from the proof of \cite[Proposition 4.5]{CEOR}.
	
	Suppose first that $\phi:\G_E\to\G_F$ is an isomorphism. We will prove that the homeomorphism $h:\partial E\to\partial F$ induced by $\phi$ is an orbit equivalence by constructing continuous functions $k,l:\partial E^{\ge 1}\to\N$ and $k',l':\partial F^{\ge 1}\to\N$ satisfying \eqref{eq:1}. 
	
	Let $e\in E^1$. Then $Z(Z(r(e)),0,1,Z(e))$ is a compact and open bisection of $\G_E$. It follows that $A_e:=\phi((Z(r(e)),0,1,Z(e)))$ is a compact and open bisection of $\G_F$. Let $x\in Z(e)$. Then $(\sigma_E(x),-1,x)\in Z(Z(r(e)),0,1,Z(e))$, so $\alpha_{A_e}(h(x))=h(\sigma_E(x))$. It follows from \cite[Proposition 3.3]{BCW} that there is an open neighborhood $V'_x$ of $h(x)$  and $k_x,l_x\in\mathbb{N}$ such that $V'_x\subseteq s(A_e)$ and $\sigma_F^{l_x}(h(x'))=\sigma_F^{k_x}(\alpha_{A_e}(h(x')))=\sigma_F^{k_x}(h(\sigma_E(x')))$ for all $x'\in h^{-1}(V'_x)$. Let $V_x:=h^{-1}(V'_x)$. Since $Z(e)$ is compact, and $\partial E$ is totally disconnected, it follows that there is a finite set $F_x$ of mutually disjoint compact and open sets such that $Z(e)=\bigcup_{B\in F_x}B$ and such that there for each $B\in F_x$ is an $x_B\in Z(e)$ such that $B\subseteq V_{x_B}$. Define functions $k_e,l_e:Z(e)\to\mathbb{N}$ by setting $k_e(x):=k_{x_B}$ and $l_e(x):=l_{x_B}$ for $x\in B$. Then $k_e$ and $l_e$ are continuous and $\sigma_F^{k_e(x)}(h(\sigma_E(x)))=\sigma_F^{l_e(x)}(h(x))$ for all $x\in Z(e)$. By doing this for each $e\in E^1$ we get continuous functions $k,l:\partial E^{\ge 1}\to\N$ satisfying \eqref{eq:1}. 
	
	Continuous functions  $k',l':\partial F^{\ge 1}\to\N$ satisfying \eqref{eq:1} can be constructed in a similar way. Thus, $h$ is an orbit equivalence. Since $x$ is eventually periodic if and only if $\{\eta\in\G_E:r(\eta)=s(\eta)=x\}$ is infinite, and $h(x)$ is eventually periodic if and only if $\{\eta\in\G_F:r(\eta)=s(\eta)=h(x)\}$ is infinite, it follows that $h$ maps eventually period points to eventually periodic points. Since, $h$ is a homeomorphism, it follows that $h$ preserves isolated eventually periodic points.
	
	For the converse, suppose $h:\partial E\to\partial F$ is an orbit equivalence that preserves isolated eventually periodic points. We will construct an isomorphism $\phi:\G_E\to\G_F$ that induces $h$. We first define $\phi(\eta)$ when $s(\eta)$ is not an isolated eventually periodic point. 
	
	Let $\eta\in\G_E$ and suppose $s(\eta)$ is not an isolated eventually periodic point. Then $h(s(\eta))$ is not an isolated eventually periodic point by assumption. Choose an open bisection $A$ such that $\eta\in A$. According to the proof of \cite[Proposition 3.4]{BCW}, the homeomorphism $h\circ\alpha_A\circ h^{-1}:h(s(A))\to h(r(A))$ belongs to $\mathcal{P}_F$. It follows from Lemma~\ref{lem:1} that there is a unique $n\in\mathbb{Z}$ with the property that there is an open neighbourhood $U$ of $h(s(\eta))$ and $k,l\in\mathbb{N}$ such that $U\subseteq h(s(A))$, $n=l-k$, and $\sigma_F^k(h(\alpha_A(h^{-1}(y))))=\sigma_F^l(y)$ for all $y\in U$. This $n$ does not depend of the choice of the open bisection $A$ because if $A'$ is another open bisection containing $\eta$, and $n'\in\mathbb{Z}$ has the property that there is an open neighbourhood $U'$ of $h(s(\eta))$ and $k',l'\in\mathbb{N}$ such that $U'\subseteq h(s(A'))$, $n'=l'-k'$, and $\sigma_F^{k'}(h(\alpha_{A'}(h^{-1}(y))))=\sigma_F^{l'}(y)$ for all $y\in U'$, then $A\cap A'$ is an open bisection containing $\eta$ and $\alpha_A(x')=\alpha_{A\cap A'}(x')=\alpha_{A'}(x')$ for all $x'\in h^{-1}(U\cap U')$, so it follows from the uniqueness of $n$ that $n'=n$. We set $\phi(\eta):=(h(r(\eta)),n,h(s(\eta)))$.
	
	We then define $\phi(\eta)$ when $s(\eta)$ is an isolated eventually periodic point. For each isolated eventually periodic point $x\in \partial E$, let $[x]:=\{x'\in\partial E:\exists \eta'\in\G_E\text{ such that }r(\eta')=x\text{ and }s(\eta')=x'\}$. Then every $x'\in [x]$ is an isolated eventually periodic point, and $[x]$ contains a periodic point. 
	
	 For each equivalence class $W\in\{[x]:x\text{ is an isolated eventually periodic point in }X\}$, choose a periodic point $x_W\in W$. For $x\in W$, let $j_x:=\min\{j\in\mathbb{N}:\sigma_E^j(x)=x_W\}$. If $\eta'\in\G_E$ and $s(\eta')$ is an isolated eventually periodic point, then so is $r(\eta')$. Furthermore, $[r(\eta')]=[s(\eta')]$ and $j_{r(\eta')}-j_{s(\eta')}-c_E(\eta')=n\lp(x_{s(\eta')})$ for some $n\in\mathbb{Z}$. We write $n_{\eta'}$ for this $n$. 
	
	Similarly, for each isolated eventually periodic point $y\in \partial F$, let $[y]:=\{y'\in\partial F:\exists \eta'\in\G_F\text{ such that }r(\eta')=y\text{ and }s(\eta')=y'\}$, choose for each equivalence class $W'\in\{[y]:y\text{ is an isolated eventually periodic point in }Y\}$ a periodic point $y_{W'}\in W'$, and let for $y\in W'$, $j_y:=\min\{j\in\mathbb{N}:\sigma_F^j(y)=y_{W'}\}$. 
	
	If $\eta\in\G_E$ and $s(\eta)$ is an isolated eventually periodic point, then so is $r(\eta)$. It follows by assumption that $h(s(\eta))$ and $h(r(\eta))$ are isolated eventually periodic points, and 
$$\bigl(h(r(\eta)),j_{h(r(\eta))}-j_{h(s(\eta))}-n_\eta\lp(y_{r(\eta)}),h(s(\eta)\bigr)\in\G_F.$$ 
We set $\phi(\eta):=(h(r(\eta)),j_{h(r(\eta))}-j_{h(s(\eta))}-n_\eta\lp(y_{r(\eta)}),h(s(\eta))$.
	
	We have now constructed a map $\phi:\G_E\to\G_F$ such that $s(\phi(\eta))=h(s(\eta))$ and $r(\phi(\eta))=h(r(\eta))$ for all $\eta\in\G_E$. It is routine to check that $\phi$ is bijective and a groupoid homomorphism. That $\phi$ and $\phi^{-1}$ are continuous can be proved similarly to how it is proved that $\phi$ is continuous in the proof of \cite[Proposition 4.5]{CEOR}. Thus, $\phi$ is an isomorphism that induces $h$.
\end{proof}

\begin{remark}
	Let $E$ and $F$ be directed graphs and $h:\partial E\to\partial F$ an orbit equivalence that preserves isolated eventually periodic points. By studying the proof of Theorem~\ref{thm:b}, one sees that there is a unique isomorphism $\phi:\G_E\to\G_F$ that induces $h$ if and only if there are no isolated eventually periodic point in $\partial E$ (i.e, if and only if $E$ satisfies condition (L)), in which case there are no isolated eventually periodic point in $\partial F$ either. 
	
\end{remark}


As in \cite{CRS}, we denote by $SE$ the directed graph obtained by attaching a head to every vertex of a directed graph $E$ (see \cite[Definition 4.2]{Tomforde}).

\begin{corollary}\label{cor:1}
	Suppose $E$ and $F$ are directed graphs with finitely many vertices and no sinks. Then any homeomorphism between $\partial E$ and $\partial F$ and any homeomorphism between $\partial SE$ and $\partial SF$ preserves isolated eventually periodic points. Thus, $\G_E$ and $\G_F$ are isomorphic if and only if $E$ and $F$ are orbit equivalent, and $\G_{SE}$ and $\G_{SF}$ are isomorphic if and only if $SE$ and $SF$ are orbit equivalent.
\end{corollary}

\begin{proof}
	Since $E$ and $F$ have finitely many vertices and no sinks, $\partial E$, $\partial F$, $\partial SE$, and $\partial SF$ contain no isolated finite paths, and no wandering paths. It therefore follows from Proposition~\ref{pro:a} that every isolated point in $\partial E$, $\partial F$, $\partial SE$, and $\partial SF$ is eventually periodic. Since a homeomorphism maps isolated points to isolated points, it follows that any homeomorphism between $\partial E$ and $\partial F$ and any homeomorphism between $\partial SE$ and $\partial SF$ preserves isolated eventually periodic points. The rest of the corollary then follows from Theorem~\ref{thm:b}.
\end{proof}


\begin{thebibliography}{12}
	

\bibitem{ABHS} P. Ara, J. Bosa, R. Hazrat and A. Sims, \emph{Reconstruction of graded groupoids from graded Steinberg algebras}, to appear in Forum Math., doi:10.1515/forum-2016-0072, 15 pages.

\bibitem{AER} S.E. Arklint, S. Eilers and E. Ruiz, \emph{A dynamical characterization of diagonal preserving $*$-isomorphisms of graph $C^*$-algebras}, \texttt{arXiv:1605.01202v1}, 19 pages.

\bibitem{BCH} J.H. Brown, L. Clark, and A. an Huef, \emph{Diagonal-preserving ring $*$-isomorphisms of Leavitt path algebras}, \texttt{arXiv:1510.05309v3}, 24 pages.

\bibitem{BCW} N. Brownlowe, T.M. Carlsen and M.F. Whittaker, \emph{Graph algebras and orbit equivalence}, to appear in Ergodic Theory Dynam. Systems, doi:10.1017/etds.2015.52, 29 pages.

\bibitem{C} T.M. Carlsen, \emph{$*$-isomorphisms of Leavitt path algebras over $\mathbb{Z}$}, \texttt{arXiv:1601.00777v2}, 8 pages.

\bibitem{CEOR} T.M Carlsen, S. Eilers, E. Ortega and G. Restorff, {\em Flow equivalence and orbit equivalence for shifts of finite type and isomorphism of their groupoids}, \texttt{arXiv:1610.09945v3}, 23 pages.

\bibitem{CR} T.M. Carlsen and J. Rout \emph{Diagonal-preserving gauge-invariant isomorphisms of graph $C^*$-algebras}, \texttt{arXiv:1610.00692v1}, 12 pages.

\bibitem{CRS} T.M. Carlsen, E. Ruiz and A. Sims, \emph{Equivalence and stable isomorphism of groupoids, and diagonal-preserving stable isomorphisms of graph $C^*$-algebras and Leavitt path algebras}, to appear in Proc. Amer. Math. Soc., doi:10.1090/proc/13321, 12 pages.

\bibitem{CS} L.O. Clark and A. Sims, \emph{Equivalent groupoids have Morita equivalent Steinberg algebras}, J. Pure Appl. Algebra \textbf{219} (2015), 2062–-2075.

\bibitem{Mat} K. Matsumoto, {\em Orbit equivalence of topological Markov shifts and Cuntz-Krieger algebras}, Pacific J. Math. {\bf 246} (2010), no. 1, 199--225.

\bibitem{Renault} J.N. Renault, \emph{Cartan subalgebras in $C^*$-algebras}, Irish Math. Soc. Bull. {\bf 61} (2008), 29--63.

\bibitem{Tomforde} M. Tomforde, \emph{Stability of $C^*$-algebras associated to graphs}, Proc. Amer. Math. Soc. {\bf 132} (2004), 1787-–1795.

\bibitem{Web} S. Webster, {\em The path space of a directed graph},  Proc. Amer. Math. Soc. {\bf 142} (2014), 213--225.

\end{thebibliography}
\end{document}